\documentclass[11pt]{amsart}

\usepackage{fullpage}

\usepackage[usenames,dvipsnames,svgnames,table]{xcolor}

\usepackage{amsmath, amssymb, mathrsfs}

\usepackage{tikz}
 \usetikzlibrary{cd}
 \tikzset{
  symbol/.style={
    draw=none,
    every to/.append style={
      edge node={node [sloped, allow upside down, auto=false]{$#1$}}}
      }
      }
 \usetikzlibrary{math}

\usepackage{amsthm}
	\theoremstyle{definition} 
	\newtheorem{definition}{Definition}[section]
	
	\theoremstyle{plain} 
	\newtheorem{theorem}[definition]{Theorem}
	\newtheorem{theorem*}{Theorem} 
	\newtheorem{lemma}[definition]{Lemma}
	\newtheorem{proposition}[definition]{Proposition}
	\newtheorem{corollary}[definition]{Corollary}
	
	\theoremstyle{remark} 
	\newtheorem{remark}[definition]{Remark}
        
\usepackage{enumitem}
\usepackage[colorinlistoftodos]{todonotes}

\newcommand{\CC}{\mathbb C}

\newcommand{\QQ}{{\mathbb Q}}
\newcommand{\RR}{{\mathbb R}}

\newcommand{\ZZ}{{\mathbb Z}}

\newcommand{\cF}{{\mathcal F}}
\newcommand{\cG}{{\mathcal G}}

\newcommand{\cO}{{\mathcal O}}
\newcommand{\cP}{{\mathcal P}}

\newcommand{\cS}{{\mathcal S}}

\newcommand{\cU}{{\mathcal U}}
\newcommand{\cV}{{\mathcal V}}

\newcommand{\cY}{{\mathcal Y}}

\newcommand{\frakS}{\mathfrak S}

\newcommand{\frakY}{\mathfrak Y}

\newcommand{\frakc}{\mathfrak c}

\newcommand{\frakp}{\mathfrak p}

\newcommand{\Qbar}{\overline{\QQ}}

\newcommand{\Qpbar}{\Qbar_p}

\usepackage{colonequals}
\newcommand{\defeq}{\colonequals} 

\newcommand{\isom}{\cong} 

\newcommand{\inj}{\hookrightarrow}

\newcommand{\tensor}{\otimes} 
\newcommand{\Directsum}{\bigoplus} 

\DeclareMathOperator\Hom{Hom} 
\DeclareMathOperator\End{End} 
\DeclareMathOperator\GL{GL} 

\DeclareMathOperator\Spec{Spec} 

\newcommand{\univ}{\mathrm{univ}} 

\DeclareMathOperator\Res{Res} 

\newcommand{\pr}{\mathrm{pr}} 

\DeclareMathOperator\val{val} 
\DeclareMathOperator\Gal{Gal} 
\DeclareMathOperator\Tr{Tr} 
\let\det\relax
\DeclareMathOperator{\det}{det} 

\DeclareMathOperator\Fitt{Fitt}

\usepackage[OT2,T1]{fontenc}
\DeclareSymbolFont{cyrletters}{OT2}{wncyr}{m}{n}
\DeclareMathSymbol{\Sha}{\mathalpha}{cyrletters}{"58} 

\DeclareRobustCommand{\SkipTocEntry}[5]{} 
\usepackage[bookmarks=false, 
        colorlinks, 
        linkcolor=red, 
        anchorcolor=red, 
        citecolor=red, 
        urlcolor=red] 
        {hyperref}  

\newcounter{counter}
\begin{document}
\title{Partial Classicality of Hilbert Modular Forms}
\author{Chi-Yun Hsu}
\address{Department of Mathematics, University of California, Los Angeles, Los Angeles, CA 90095, USA}
\email{cyhsu@math.ucla.edu}
\date{\today}


\begin{abstract}

Let $F$ be a totally real field and $p$ a rational prime unramified in $F$. We prove a partial classicality theorem for overconvergent Hilbert modular forms: when the slope is small compared to a subset of weights, an overconvergent form is partially classical.
We use the method of analytic continuation.
 
\end{abstract}

\maketitle

\tableofcontents


\section{Introduction}
Coleman \cite{Col96} proved that a $p$-adic overconvergent modular form of weight $k\in \ZZ$ must be classical if its slope, i.e., the $p$-adic valuation of the $U_p$-eigenvalue, is less than $k-1$.
His proof involves analyzing the rigid cohomology of modular curves.
On the other hand, Buzzard \cite{Buz} and Kassaei \cite{Kas_an} developed the alternate method of analytic continuation to prove classicality theorems.
The key is to understand the dynamic of the $U_p$ Hecke operator.

Let $F$ be a totally real field of degree $g$ over $\QQ$.
In the situation of Hilbert modular forms associated to $F$, many results about classicality are also known. 
Coleman's cohomological method was developed by Tian--Xiao \cite{TX16} to prove a classicality theorem, assuming $p$ is unramified in $F$.
The method of analytic continuation was worked out first in the case when $p$ splits completely in $F$ by Sasaki \cite{Sasaki}, then in the case when $p$ is unramified by Kassaei \cite{Kas} and Pilloni--Stroh \cite{PS}, and finally when $p$ is allowed to be ramified by Bijakowski \cite{Bij}.

Let $\Sigma$ be the set of archimedean embeddings of $F$, which we identify with the set of $p$-adic embeddings of $F$ through some fixed isomorphism $\CC\isom \Qpbar$.
For each prime $\frakp$ of $F$ above $p$, denote by $\Sigma_\frakp\subseteq \Sigma$ the subset of $p$-adic embeddings inducing $\frakp$.
Let $e_\frakp$ be the ramification index, and $f_\frakp$ the residue degree of $\frakp$.
Then the classicality theorem for overconvergent Hilbert modular forms proved by analytic continuation is as follows.

\begin{theorem*}[Bijakowski]
Let $f$ be an overconvergent Hilbert modular form of weight $\underline{k}\in \ZZ^{\Sigma}\isom \ZZ^g$.
Assume that for all $\frakp\mid p$, $U_\frakp(f) = a_\frakp f$ such that 
\[
\val_p(a_\frakp) < \frac1{e_\frakp}\inf_{\tau\in  \Sigma_\frakp} \{k_\tau\}-f_\frakp,
\]
where $\val_p$ is the $p$-adic valuation normalized so that $\val_p(p)=1$.
Then $f$ is classical.
\end{theorem*}
\begin{remark}
When $p$ is unramified in $F$, namely $e_\frakp=1$ for all $\frakp\mid p$, Tian--Xiao proved the classicality theorem with weaker slope assumption: $\val_p(a_\frakp) < \inf_{\tau\in  \Sigma_\frakp} \{k_\tau\}-1$. This slope bound is believed to be optimal \cite[Proposition 4.3]{Breuil}.

\end{remark}

In this paper, we prove some ``partial'' classicality theorems for overconvergent Hilbert modular forms.
Let $I$ be a subset of $\Sigma$.
Breuil defined the notion of \emph{$I$-classical} overconvergent Hilbert modular forms (see \cite[p.~3]{Breuil} or Definition \ref{def:I-cl}).
When $I=\varnothing$, they are the usual overconvergent forms; when $I=\Sigma$, they are the classical forms.

\begin{theorem*}[Theorem~\ref{thm:partial_classicality}] \label{thm:main}
Assume that $p$ is unramified in $F$.
Let $f$ be an overconvergent Hilbert modular form of weight $\underline{k}\in \ZZ^{\Sigma}$.
Let $I\subseteq \Sigma$.
Assume that for all $\frakp\mid p$, $U_\frakp(f) = a_\frakp f$ such that 
\[
\val_p(a_\frakp) < \inf_{\tau\in I\cap \Sigma_\frakp} \{k_\tau\}-f_\frakp.
\]
Then $f$ is $I$-classical.
\end{theorem*}

We use the method of analytic continuation to prove Theorem~\ref{thm:main}.
In the situation when $I=\Sigma$, this recovers the classicality theorem proven by Kassaei or Pilloni--Stroh, who assumed $p$ is unramified.
Although when $I=\Sigma$, Bijakowski proved a classicality theorem not assuming $p$ is unramified, it is Kassaei's approach that is more suitable for partial classicality.
Indeed, when studying the dynamic of $U_\frakp$-operators, it has been proven to be successful to use \emph{degree} to parametrize regions on the Hilbert modular variety, and analyze how $U_\frakp$-operators influence degrees.
Kassaei made efforts to analyze how $U_\frakp$-operators affect the more refined \emph{directional degrees}, but only when $p$ is unramified.
On the other hand, Bijakowski was able to use only the degree function to prove a classicality theorem allowing $p$ to be ramified.
In the situation of partial classicality, the weight $k_\tau$ with $\tau\in \Sigma$ in the slope condition is independent of each other, while the $U_\frakp$-operator intertwines all directional degrees inducing $\frakp$.
As a result, we cannot avoid analyzing the directional degrees like Bijakowski did.

We mention some related work on partial classicality theorems.
Barrera Salazar and Williams \cite{BS-W} took the perspective of overconvergent cohomology for a general quasi-split reductive group $\cG$ over $\QQ$ with respect to a parabolic subgroup $Q$ of $G=\cG/\QQ_p$.
Applying their work to the situation of Hilbert modular forms (i.e., $\cG = \Res_{F/\QQ}\GL_2$), we would recover Theorem \ref{thm:main} in the restrictive case of $I\subseteq \Sigma$ such that $I\cap \Sigma_\frakp$ is either $\Sigma_\frakp$ or $\varnothing$ for each $\frakp\mid p$.
In \cite[Proposition 4.3(i)]{Breuil} for the special case $S_{>1} = S$, Breuil gave a conjecture about partial classicality:
In the restrictive case when $I$ is either $\Sigma_\frakp$ or $\varnothing$ for each $\frakp\mid p$, if $f$ satisfies the weaker assumption $\val_p(a_\frakp) < \inf_{\tau\in \Sigma_\frakp} \{k_\tau\} -1$ for all $\frakp\mid p$ such that $I\cap \Sigma_\frakp\neq \varnothing$, then $f$ is $I$-classical.
Yiwen Ding \cite[Appendix A]{Ding} studied partial classicality from the perspective of Galois representations.
He also did not restrict to the case when $I\cap \Sigma_\frakp$ is either $\Sigma_\frakp$ or $\varnothing$.
Namely, let $\rho_f\colon \Gal_F\rightarrow \GL_2(L)$ be the Galois representation associated to an overconvergent Hilbert Hecke eigenform $f$, where $L$ is a finite extension of $\QQ_p$. If $\val_p(a_\frakp)<\inf_{\tau\in I\cap \Sigma_\frakp}\{k_\tau\}-1$, then $\left.\rho_f\right|_{\Gal_{F_\frakp}}$ is $(I\cap \Sigma_\frakp)$-de Rham.

There are many interesting questions related to $I$-classical overconvergent forms.
In the direction of classicality, can we prove Theorem~\ref{thm:main} not assuming $k_\tau$ is an integer for $\tau\notin I$? 
If $f$ is $I$-classical and $\val_p(a_\frakp)<\inf_{\tau\in J\cap \Sigma_\frakp}\{k_\tau\}-f_\frakp$, will $f$ be $I\cup J$-classical 
(see \cite[Conjecture 3.2 (ii)]{Breuil})?
Relating to Galois representations, 
if $f$ is an $I$-classical Hilbert Hecke eigenform, does the Galois representation $\rho_f$ satisfy the condition that $\left.\rho_f\right|_{\Gal_{F_\frakp}}$ is $(I\cap \Sigma_\frakp)$-de Rham for all $\frakp\mid p$?
If this is true, one can further ask in the flavor of Kisin's interpretation of Fontaine--Mazur conjecture: if $f$ is overconvergent and $\left.\rho_f\right|_{\Gal_{F_\frakp}}$ is $(I\cap \Sigma_\frakp)$-de Rham for all $\frakp\mid p$, is $f$ $I$-classical?

For the organization of this paper: In Section \ref{sec:Hilb}, we define the degree function and partially classical overconvergent forms.
In Section \ref{sec:classicality}, we prove Theorem \ref{thm:main}.

\subsection*{Notations}
Fix a totally real field $F$ of degree $g$ over $\QQ$.
Let $\Sigma$ denote the set of archimedean places of $F$; in particular $\#\Sigma=g$.
Fix a rational prime $p$ which is unramified in $F$ and $(p) = \frakp_1 \cdots \frakp_r$ in $F$.
For each prime $\frakp$ of $F$ above $p$, let $f_\frakp$ be the residue degree of $\frakp$.
Fix an isomorphism $\iota_p\colon\CC\xrightarrow{\sim}\Qpbar$, and identify archimedean embeddings $\tau\colon F \rightarrow \CC$ with $p$-adic embeddings $\iota_p\circ \tau \colon F \rightarrow \Qpbar$.
For each prime $\frakp$ of $F$ above $p$, let $\Sigma_{\frakp}\subseteq \Sigma$ be the subset of $p$-adic embeddings inducing $\frakp$.
Hence $\#\Sigma_\frakp = f_\frakp$.
Let $L$ be a finite extension of $\QQ_p$ containing the image of all $p$-adic embeddings $\iota\circ \tau$ of $F$.
Since $p$ is assumed to be unramified in $F$, we may also assume that $L$ is an unramified extension of $\QQ_p$.
Let $k_L$ denote the residue field of $L$.
Let $\delta_F$ be the different ideal of $F$.

\section{Partially classical overconvergent forms}
\label{sec:Hilb}


\subsection{Hilbert modular varieties} 
Let $N\geq4$ be an integer, and $p\nmid N$.
Let $\frakc$ be a fractional ideal of $F$.
Denote by $\frakc^+\subseteq\frakc$ the cone of totally positive elements, i.e., the elements in $\frakc$ which are positive under every embedding $\tau\colon F\rightarrow \RR$.
Let $Y_\frakc\rightarrow \Spec \cO_L$ be the Hilbert modular scheme classifying $(\underline{A}, H) = (A/S,i,\lambda,\alpha, H)$ where 
\begin{itemize}
    \item $A$ is an abelian scheme of relative dimension $g$ over an $\cO_L$-scheme $S$,
    
    \item $i\colon \cO_F \hookrightarrow \End_S(A)$ is a ring homomorphism. Here $i$  is called a real multiplication on $A$,
    
    \item $\lambda\colon (\cP_A, \cP_A^+) \rightarrow (\frakc,\frakc^+)$ is an isomorphism of $\cO_F$-modules identifying the positive elements, and inducing an isomorphism $A\tensor_{\cO_F} \frakc \isom A^\vee$.
    Here $\cP_A=\Hom_{\cO_F}(A,A^\vee)^\mathrm{sym}$ is the projective $\cO_F$-module of rank $1$ consisting of symmetric morphisms from $A$ to its dual abelian scheme $A^\vee$, and $\cP_A^+\subseteq \cP_A$ is the cone of polarizations.
    Here $\lambda$ is called a $\frakc$-polarization of $A$,
   
    \item$\alpha\colon \mu_N\tensor \delta_F^{-1} \rightarrow A$ is a closed immersion of $\cO_F$-group schemes. 
    Here $\alpha$ is called a $\Gamma_1(N)$-level structure, and 
    \item $H\subseteq A[p]$ is a finite flat $\cO_F$-subgroup scheme of rank $p^g$ which is isotropic with respect to the $\mu$-Weil pairing for some polarization $\mu\in \cP_A^+$ of degree prime to $p$.
\end{itemize} 
Let $\mathrm{Cl}(F)^+$ be the narrow class group of $F$, namely the quotient of the abelian group of fractional ideals of $F$ by the subgroup of principal ideals generated by totally positive elements. 
Let $\{c_i\}$ be a set of representatives of $\mathrm{Cl}(F)^+$.
Define $Y=\coprod_{\frakc_i} Y_{\frakc_i}$, which is independent of the choice of the representatives $\{c_i\}$.
Denote by $\frakY$ the completion of $Y$ along its special fiber, and by $\cY$ the rigid generic fiber of the formal scheme $\frakY$.
We also use this convention of letter styles for other schemes: when $K/\QQ_p$ is a finite extension and $S$ is a scheme over $\cO_K$, we denote by $\frakS$ the associated formal scheme and by $\cS$ the rigid generic fiber of $\frakS$.

\subsection{Directional degrees}
We first recall the definition of the degree for a commutative finite flat group scheme.
See \cite{Fargue} for more detailed studies of the concept.

Let $S$ be a scheme and $G$ a commutative finite flat group scheme over $S$.
Let $\omega_G$ be the sheaf of invariant differentials on $G$.
Define 
\[ \delta_G \defeq \Fitt_0 \omega_G \]
as the $0$-th Fitting ideal of $\omega_G$.
This is an invertible ideal sheaf in $\cO_S$.

Now let $K/\QQ_p$ be a finite extension and $S=\Spec \cO_K$.
Then the degree of $G$ is defined as \cite[D\'efinition 4]{Fargue} the rational number
\[\deg G = \deg \omega_G \defeq \val_p (\delta_G).\]
Writing $\omega_G = \Directsum_i \cO_K/x_i\cO_K$, then $\deg G = \sum_i \val_p(x_i)$.
Equivalently, $\deg G =\ell(\omega_G)/e_K$,
where $\ell(\omega_G)$ is the length of the $\cO_K$-module $\omega_G$, and $e_K$ is the ramification index of $K$.
Recall that the height $\operatorname{ht} G$ of $G$ is such that $|G|=p^{\operatorname{ht} G}$.
Hence $G$ is \'etale if and only if $\deg G=0$, and $G$ is multiplicative if and only if $\deg G = \operatorname{ht} G$.

More generally, let $S$ be a scheme over $\cO_K$. Each closed point $s$ in the rigid analytic space $\cS$ is defined over the ring of integer of a finite extension of $K$ \cite[Section 8.3, Lemma 6]{Bosch}.
Hence we obtain the degree function $$\deg\colon \cS \rightarrow [0,\infty) \cap \QQ \quad s\mapsto \deg G_s.$$
The inverse image of a (open, closed, or half-open) interval  in $[0,\infty)$ is an admissible open of $\cS$.
Moreover, when the interval is closed and its end points $a\leq b$ are rational numbers, then the inverse image is quasi-compact.

We record some properties of $\deg$ which we will constantly use for computation.
\begin{lemma}\cite[lemme 4]{Fargue}
Let $0\rightarrow G'\rightarrow G\rightarrow G''\rightarrow 0$ be a short exact sequence of finite flat group schemes over $S$.
Then $\deg G = \deg G'+\deg G''$.
\end{lemma}

\begin{lemma} \label{lem2} \cite[p.~2]{Fargue}
Let $\lambda\colon A\rightarrow B$ be an isogeny of $p$-power degree between abelian schemes over $S$.
Let $G\defeq \ker \lambda$.
Let $\omega_{A/S}$ and $\omega_{B/S}$ be the sheaves of invariant differentials of $A$ and $B$, respectively.
Let $\lambda^\ast\colon \omega_{B/S} \rightarrow \omega_{A/S}$ be the induced pullback map.
Then 
\[
\deg G = \val_p(\det \lambda^\ast).\]
In particular, if $A$ is of dimension $g$, then $\deg A[p] = g$.
\end{lemma}

When $G$ has an $\cO_F$-module structure, we can define directional degree functions on $\cS$.
Instead of a general exposition, we only explain this for $\cS = \cY$, the Hilbert modular variety over $L$.
See also \cite[Section 4.2]{PS} or \cite[Section 2.9]{Kas}.
Let $(\underline{A}^\univ, H^\univ)$ be the universal abelian scheme over the Hilbert modular variety $Y$.
Let $\omega_{H^\univ}$ be the sheaf of invariant differentials of $H^\univ$, which is a $\cO_F/p\cO_F$-module.
Since $p$ is unramified in $F$, $\Sigma$ is in bijection with the embeddings $\cO_F/p\cO_F \inj k_L$.
We decompose $\omega_{H^\univ}$ according to the embeddings $\cO_F/p\cO_F \inj k_L$ to obtain $$\omega_{H^\univ} = \bigoplus_{\tau\in \Sigma} \omega_{H^\univ,\tau}.$$
For each $\tau\in \Sigma$, define $\delta_\tau \defeq \Fitt_0(\omega_{H^\univ,\tau})$, which is an invertible ideal sheaf in $\cO_Y$.

Let $y = (\underline{A},H)$ be a closed point of $\cY$.
Let $K$ be the finite extension of $L$ over which $y$ is defined.
Then we have the rational number $\deg \omega_{H,\tau}$.
In addition, $\deg \omega_{H,\tau} \in [0,1]$.
Indeed, for each $\frakp\mid p$, the subgroup scheme $H[\frakp]$ of $H$ is a Raynaud group scheme over $\Spec \cO_K$, namely a $k_\frakp=\cO_F/\frakp\cO_F$-vector space scheme of dimension $1$.
For each tuple $(d_\tau)_{\tau\in \Sigma_\frakp}$ of elements of $\cO_K$ with $\val_p (d_\tau) \leq 1$, Raynaud associates a $k_\frakp$-vector space scheme of dimension $1$ 
\[ H_{(d_\tau)} \defeq \Spec \cO_K[X_\tau, \tau\in \Sigma_\frakp]/ (X_{\sigma^{-1}\circ\tau}^p-d_\tau X_\tau), \]
where $\sigma$ is the Frobenius automorphism of $L$ over $\QQ_p$ lifting $x\mapsto x^p$ modulo $p$, and the $k_\frakp$-action on $X_\tau$ is given by the character $k_\frakp^\times \rightarrow \cO_K^\times$ induced by 
 $\tau \colon F \rightarrow L$.
Moreover, each $k_\frakp$-vector space scheme of dimension $1$ over $\cO_K$ is isomorphic to some $H_{(d_\tau)}$  \cite[TH\'EOR\`EME 1.4.1]{Raynaud}.
Since $\omega_{H_{(d_\tau)},\tau} = \cO_K/d_\tau\cO_K$, we have $\deg \omega_{H,\tau} = \deg \omega_{H_{(d_\tau)},\tau} = \val_p(d_\tau) \in [0,1]$.

Hence for each $\tau\in \Sigma$, we can define
the directional degree function
\[\deg_\tau\colon \cY \rightarrow [0,1]\cap \QQ\quad y = (\underline{A},H) \mapsto \deg \omega_{H,\tau},
\]
as well as 
\[\underline{\deg}\colon \cY \rightarrow ([0,1]\cap \QQ)^\Sigma  \quad y \mapsto (\deg_\tau y)_\tau.
\]
%
As before, the inverse image of $\deg_\tau$ (resp. $\underline{\deg}$) of a subset of $[0,1]$ (resp.  $[0,1])^\Sigma$) defined by a finite number of affine inequalities is an admissible open of $\cY$.
Moreover, when the inequalities are all non-strict and the coefficients are all rational numbers, then the inverse image is quasi-compact.


Given $I\subseteq \Sigma$, we define 
\begin{align*}
\cF_I 
&\defeq \prod_{\tau\in\Sigma}\cF_{I,\tau}, \text{ where } \cF_{I,\tau} = \begin{cases} [0,1], & \tau\in I \\ [1,1], & \tau \not\in I.\end{cases} 
\end{align*}
Then $\cF_I$ is a closed $|I|$-dimensional hypercube in $([0,1]\cap \QQ)^{\Sigma} = {\cF}_{\Sigma}$.
We also define $x_I\in [0,1]^\Sigma$ to be the vertex
\[
x_{I,\tau} = \begin{cases} 0, & \tau\in I\\ 1, & \tau \not\in I. \end{cases}
\]
Hence the vertices of $\cF_I$ are exactly the $x_J$'s with $J\subseteq I$.
Denote by $\cY \cF_I$ the quasi-compact admissible open $\underline{\deg}^{-1} \cF_I$ of $\cY$.

\begin{definition}
Let $\frakp\mid p$ be a prime of $F$.
For $\tau\in \Sigma_\frakp$, define the \emph{twisted directional degree} 
\[\tilde\deg_\tau\colon \cY \rightarrow [0,\frac{p^{f_\frakp}-1}{p-1}]\cap \QQ\]
by
\[
\tilde\deg_\tau \defeq \sum_{j=0}^{f_\frakp-1} p^{f_\frakp-1-j} \deg_{\sigma^j\circ\tau} = p^{f_\frakp-1} \deg_\tau + p^{f_\frakp-2} \deg_{\sigma\circ\tau}+\cdots +\deg_{\sigma^{f_\frakp-1}\circ\tau}.
\]
Here $\sigma$ is the Frobenius automorphism of the unramified extension $L$ over $\QQ_p$, lifting $x\mapsto x^p \bmod p$.
We also define 
\[\underline{\tilde\deg}\colon \cY \rightarrow ([0,\frac{p^{f_\frakp}-1}{p-1}]\cap \QQ)^\Sigma  \quad y \mapsto (\tilde\deg_\tau y)_\tau.
\]
\end{definition}

We use the overhead tilde notation $\tilde {(\cdot)}$ to denote the image under the linear transformation 
\[
\RR^\Sigma \rightarrow \RR^\Sigma  \quad (x_\tau)_\tau \mapsto (\tilde{x}_\tau)_\tau, \text{ where } \tilde x_\tau =  \sum_{j=0}^{f_\frakp-1} p^{f_\frakp-1-j} x_{\sigma^j\circ\tau} \text{ for } \tau\in\Sigma_\frakp.
\]
In particular, if $(x_\tau)_\tau = \underline{\deg}\  y $ for some $y\in \cY$, then  $(\tilde x_\tau)_\tau = \underline{\tilde \deg}\ y$.
For example, $\tilde x_I$ is the vertex of $\tilde\cF_\Sigma$ given by 
\[\tilde x_{I,\tau} = \sum_{j=0}^{f_\frakp-1} p^{f_\frakp-1-j} x_{I,\sigma^j\circ\tau} \text{ for } \tau\in \Sigma_\frakp.\]
See Figures~\ref{fig:F} and~\ref{fig:tildeF} for an example of $\mathcal F_\Sigma$ and $\tilde{{\mathcal F}}_\Sigma$.

\begin{figure}
\centering
\begin{tikzpicture}
\draw[thick](0,0)--(2,0)--(2,2)--(0,2)--(0,0);
\draw(-0.3,-0.3) node{$x_{12}=(0,0)$};
\draw(2.3,-0.3) node{$x_2=(1,0)$}; 
\draw(2.3,2.3) node{$x_{\varnothing}=(1,1)$}; 
\draw(-0.3,2.3) node{$x_1=(0,1)$}; 
\end{tikzpicture}
\caption{$\cF_\Sigma$ when $g=2$}
\label{fig:F}
\end{figure}
\begin{figure}
\centering
\tikzmath{\p = 3;}
\begin{tikzpicture}
\draw[thick](0,0)--(\p,1)--(\p+1,\p+1)--(1,\p)--(0,0);
\draw(-0.3,-0.3) node{$\tilde x_{12}=(0,0)$}; 
\draw(\p+.3,1-.5) node{$\tilde x_2=(p,1)$}; 
\draw(\p+1.3,\p+1.3) node{$\tilde x_{\varnothing}=(p+1,p+1)$}; \draw(1-.3,\p+.5) node{$\tilde x_1=(1,p)$}; 
\end{tikzpicture}
\caption{$\tilde\cF_\Sigma$ when $g=2$}
\label{fig:tildeF}
\end{figure}

\subsection{Hilbert modular forms}
Let $\omega = \omega_{A^\univ}$ be the sheaf of relative differentials of the universal abelian scheme over $Y$.
The sheaf $\omega$ is an $\cO_F\tensor_\ZZ \cO_Y$-module, locally free of rank $1$.
The $\cO_F$-module structure on $\omega$ provides the decomposition with respect to embeddings $\tau\colon F \rightarrow L$
\[
\omega = \Directsum_{\tau \in \Sigma} \omega_\tau, \]
where each $\omega_\tau$ is an $\cO_Y$-module, locally free of rank $1$.
Given $\underline{k} = (k_\tau)_{\tau\in\Sigma}\in \ZZ^{\Sigma}$,
we define an invertible sheaf on $Y$
\[\omega^{\underline{k}} = \bigotimes_{\tau\in \Sigma}  \omega_{\tau}^{k_{\tau}}.\]
We use the same notation $\omega^{\underline{k}}$ for the invertible sheaf on $\cY$ coming from analytifying $\omega^{\underline{k}}$.

The space of \emph{Hilbert modular forms of level $\Gamma_1(N)\cap \Gamma_0(p)$ and weight $\underline{k}$} is defined to be $H^0(Y,\omega^{\underline{k}})$.
By GAGA and Koecher principle, it is the same as $H^0(\cY,\omega^{\underline{k}})$ \cite[Proposition 5.1.2]{PS}.

\begin{definition} \label{def:I-cl} 
Let $I\subseteq \Sigma$.
The space of \emph{$I$-classical overconvergent Hilbert modular forms of level $\Gamma_1(N)\cap \Gamma_0(p)$ and weight $\underline{k}$} is
\[
H^{0,\dagger}(I,\omega^{\underline{k}}) \defeq \varinjlim_\cV H^0(\cV, \omega^{\underline{k}}),
\]
where $\cV$ runs through strict neighborhoods of 
$\cY\cF_I$ in $\cY$.
\end{definition}
When $I=\varnothing$, $I$-classical simply means overconvergent, and when $I=\Sigma$, $I$-classical means classical.
Whenever $J\subseteq I$, we have a map \[H^{0,\dagger}(I,\omega^{\underline{k}})\rightarrow H^{0,\dagger}(J,\omega^{\underline{k}}) \] 
given by restriction.
This is an injective map.

\subsection{$U_\frakp$-operators} \label{sec:Up}
Let $\frakp\mid p$ be a prime of $F$ above $p$ and $f_\frakp$ the residue degree of $\frakp$.

Let $Y(\frakp)\rightarrow \Spec L$ be the moduli space whose $S$-points consist of  $(\underline{A},H,H_1)$, where $(\underline{A},H)\in Y(S)$ and $H_1\subseteq A[\frakp]$ is a finite flat isotropic $\cO_F$-subgroup scheme of rank $p^{f_\frakp}$ and $H_1\neq H[\frakp]$.
We have the $U_\frakp$-correspondence of $Y\tensor_{\cO_L} L$:
\begin{center}
\begin{tikzcd}
 & Y(\frakp) \ar[ld, "p_1"'] \ar[rd, "p_2"] & \\
Y\tensor_{\cO_L}L&& Y\tensor_{\cO_L}L   
\end{tikzcd}
\end{center}
Here the projections are
$$p_1\colon  (\underline{A},H,H_1) \mapsto (\underline{A}, H), $$
and 
$$p_2\colon  (\underline{A},H,H_1) \mapsto (\underline{A}/H_1, \bar{H}), $$
where $\bar{H}$ is the image of $H$ under $A \rightarrow A/H_1$.

Let $Y(\frakp)^{\mathrm{an}}$ be the rigid analytification of  $Y(\frakp)$
\cite[Section 5.4, Corollary 5]{Bosch}, which is a rigid analytic space over $L$.
We have the induced $U_\frakp$-correspondence, $p_1$ and $p_2$ over $(Y\tensor L)^\mathrm{an}$.
Note that $(Y\tensor L)^{\mathrm{an}}$ contains $\cY$.
Let $\cY(\frakp)\defeq Y(\frakp)^\mathrm{an}\times_{(Y\tensor L)^\mathrm{an},p_1} \cY$. 
We then have the $U_\frakp$-correspondence, $p_1$ and $p_2$ over $\cY$.

Given a subset $\cU$ of $\cY$, we then obtain a subset of $\cY$
\[ U_\frakp(\cU) \defeq p_2p_1^{-1}(\cU). \]
Given two admissible opens $\cU, \cV \subseteq \cY$ such that $U_\frakp(\cV)\subseteq \cU$,  we have $U_\frakp\colon \omega^{\underline{k}}(\cU) \rightarrow \omega^{\underline{k}}(\cV)$ defined by 
\[ (U_\frakp f)(\underline{A},H) = \frac1{p^{f_\frakp}} \sum_{(\underline{A}/H_1,\bar{H})\in U_\frakp(\underline{A},H)} \pr^\ast f(\underline{A}/H_1,\bar{H}),\]
where $\pr\colon A \rightarrow A/H_1$ is the natural projection.

We record the dynamic of $U_\frakp$ with respect to the (twisted) directional degrees. 
See \cite[Proposition 5.1.4, 5.1.14]{Kas} or \cite[Proposition 4.4.1, 4.4.2]{PS}.
\begin{proposition}
\label{prop:twisted_deg_increase}
Let $y=(\underline{A},H)\in \cY$.
Let $\frakp\mid p$ be a prime of $F$ above $p$, and $y'=(\underline{A/H_1},\bar{H})\in U_\frakp(y)$.
Then 
\begin{enumerate}
    \item $\tilde \deg_\tau (y') \geq \tilde\deg_\tau (y)$ for all $\tau\in\Sigma_\frakp$, and
    \item if $$\sum_{\tau\in\Sigma_\frakp} \tilde\deg_\tau y' =\sum_{\tau\in\Sigma_\frakp} \tilde\deg_\tau y,$$ equivalently, $\sum_{\tau\in\Sigma_\frakp} \deg_\tau y' =\sum_{\tau\in\Sigma_\frakp} \deg_\tau y$, then $\deg_\tau y \in \{0,1\}$ for all $\tau\in \Sigma_\frakp$.
\end{enumerate}
\end{proposition}

\section{Partial classicality} \label{sec:classicality}
The content of this section is to prove the following partial classicality theorem.
\begin{theorem} \label{thm:partial_classicality}
Let $f$ be an overconvergent Hilbert modular form of weight $\underline{k}$.
Let $I\subseteq \Sigma$.
Assume that for all $\frakp\mid p$, $U_\frakp(f) = a_\frakp f$ such that 
\begin{align} \label{small_slope}
\val_p(a_\frakp) < \inf_{\tau\in I\cap \Sigma_\frakp} \{k_\tau\}-f_\frakp.
\end{align}
Then $f$ is $I$-classical.
\end{theorem}

\begin{remark}
In the case of $I=\Sigma$, this is a theorem of Kassaei \cite{Kas} or Pilloni--Stroh \cite{PS}. 
Although when $I=\Sigma$, Bijakowski \cite{Bij} proved a classicality theorem not assuming $p$ is unramified, it is Kassaei's approach that is more suitable for partial classicality.
Both use the idea of analytic continuation.
Kassaei made efforts to analyze how $U_\frakp$-operators affect $\deg_\tau$ for all $\tau\in \Sigma_\frakp$, but only when $p$ is unramified.
On the other hand, Bijakowski was able to use only $\deg H[\frakp]$ to prove the classicality even when $p$ is ramified.
In the situation of partial classicality, the weight $k_\tau$ with $\tau\in \Sigma$ in the slope condition is independent of each other, while the $U_\frakp$-operator intertwines all directional degrees inducing $\frakp$, so we do need to understand the directional degrees.
\end{remark}


Throughout the section, we will assume that $p$ is inert in $F$.
To prove Theorem~\ref{thm:partial_classicality} for a general unramified $p$, we can apply the same argument to each prime $\frakp\mid p$.
For example, see \cite{Sasaki} and \cite[Lemma 7.4.2]{PS}.

Now we begin to prove Theorem~\ref{thm:partial_classicality} assuming $p$ is inert in $F$; in particular, $f_p = g$. 
We will show that if $U_p(f)=a_pf$ such that $\val_p(a_p)<\inf_{\tau\in I} k_\tau-g$, then $f$ is $J$-classical for all $J\subseteq I$, and hence $f$ is $I$-classical.
We do this by induction on $|J|$.
For $|J|=0$, it simply means that $f$ is overconvergent, which is true by assumption.
Assume that $f$ is $J$-classical for all $J\subsetneq I$, say $f$ is defined on a strict neighborhood of $\cY\cF_{J} = \underline{\deg}^{-1} \cF_J$.
In particular, $f$ is defined on a strict neighborhood of $\underline{\deg}^{-1} x_J$ for all $J\subsetneq I$.

\subsection{Automatic analytic continuation} \label{subsec:auto}
In the subsection, with the assumption that the slope of $f$ is finite (but not necessarily small), we can already show that $f$ can be analytically continued to a large region in $\cY\cF_I$.

Let $I\subseteq \Sigma$ and $\epsilon>0$.
Define
\[
\cU_I(\epsilon) = \{y\in\cY\colon \sum_{\tau\in I} \tilde \deg_\tau y \geq \sum_{\tau\in I} \tilde x_{I,\tau}+\epsilon, 
\tilde\deg_\tau y \geq p^{g-2}+\cdots+1+\epsilon, \forall \tau\not\in I\}.
\]
See Figures~\ref{fig:U1} and~\ref{fig:U12} for examples of the image of ${\mathcal U}_{I}(\epsilon)$ under $\underline{\deg}$, and Figures~\ref{fig:tilde_U1} and~\ref{fig:tilde_U12} for examples of the image of ${\mathcal U}_{I}(\epsilon)$ under $\underline{\tilde{\deg}}$.

\begin{figure}
\begin{minipage}{0.48\textwidth}
\centering
\begin{tikzpicture}
 \draw[thick](0,0)--(2,0)--(2,2)--(0,2)--(0,0);
 \draw[dotted] (0.4,2)--(1.1,0);
 \draw[dotted] (2,0.4)--(0,1.1);
\fill[gray] (0.8,0.8)--(2,0.4)--(2,2)--(0.4,2);
\end{tikzpicture}
\caption{$\underline{\deg}\ \cU_1(\epsilon)$ when $g=2$}
\label{fig:U1}
\end{minipage}
\begin{minipage}{0.48\textwidth}
\centering
\begin{tikzpicture}
 \draw[thick](0,0)--(2,0)--(2,2)--(0,2)--(0,0);
 \draw[dotted] (0.5,0)--(0,0.5);
\fill[gray] (0.5,0)--(2,0)--(2,2)--(0,2)--(0,0.5);
\end{tikzpicture}
\caption{$\underline{\deg}\ \cU_\Sigma(\epsilon)$  when $g=2$}
\label{fig:U12}
\end{minipage}
\end{figure}

\begin{figure}
\begin{minipage}{0.48\textwidth}
\centering
\tikzmath{\p = 3;}
\begin{tikzpicture}
 \draw[thick](0,0)--(\p,1)--(\p+1,\p+1)--(1,\p)--(0,0);
\draw[dotted] (1.5,\p+1)--(1.5,0);
\draw[dotted] (\p+1,1.5)--(0,1.5);
\fill[gray] (1.5,1.5)--(1.5,3.15)--(\p+1,\p+1)--(3.15,1.5);
\end{tikzpicture}
\caption{$\underline{\tilde\deg}\ \cU_1(\epsilon)$ when $g=2$}
\label{fig:tilde_U1}
\end{minipage}
\begin{minipage}{0.48\textwidth}
\centering
\tikzmath{\p = 3;}
\begin{tikzpicture}
 \draw[thick](0,0)--(\p,1)--(\p+1,\p+1)--(1,\p)--(0,0);
\draw[dotted] (0,0.65)--(0.65,0);
\fill[gray] (0.5,0.15)--(\p,1)--(\p+1,\p+1)--(1,\p)--(0.15,0.5);
\end{tikzpicture}
\caption{$\underline{\tilde\deg}\ \cU_\Sigma(\epsilon)$ when $g=2$}
\label{fig:tilde_U12}
\end{minipage}
\end{figure}

Because $\cU_I(\epsilon)$ is defined by a finite number of affine inequalities with $\tilde \deg_\tau$ (equivalently, with $\deg_\tau$), we know that $\cU_I(\epsilon)$ is an admissible open  of $\cY$.
Note that whenever $\epsilon'<\epsilon$, we have $\cU_I(\epsilon')\supseteq \cU_I(\epsilon)$.

Let $f$ be an overconvergent Hilbert modular form of weight $\underline{k}$.
Assume that $U_p(f) = a_p f$ with $\val_p(a_p)<\infty$.

\begin{lemma} \label{lem:loc_u}
Let $I\subseteq \Sigma$.
Suppose that $f$ is defined on a strict neighborhood of $\underline{\deg}^{-1} x_J=\underline{\tilde\deg}^{-1} \tilde x_J$ for all $J\subsetneq I$.
Then $f$ can be extended to $\cU_I(\epsilon)$
for any rational number $\epsilon>0$.
\end{lemma}
\begin{proof}
First of all, note that $\cU_I(\epsilon)$ is $U_p$-stable because $U_p$ increases twisted directional degrees (Proposition~\ref{prop:twisted_deg_increase}(1)).

By Proposition~\ref{prop:twisted_deg_increase}(2), $U_p$ strictly increases $\sum_{\tau\in\Sigma} \tilde\deg_\tau$ except at points $y\in \cY$ such that $\underline\deg \ y\in \{0,1\}^g$, i.e., $\underline\deg \ y= x_J$ for some $J\subseteq \Sigma$. 
Suppose that $y\in\cU_I(\epsilon)$ satisfies $\underline\deg \ y = x_J$.
We claim that $J\subsetneq I$.
Indeed, for $\tau\in J$, $\tilde\deg_\tau y\leq p^{g-2}+\cdots+1$.
Hence the second condition of $\cU_I(\epsilon)$
\[\tilde\deg_\tau y \geq p^{g-2}+\cdots+1+\epsilon, \forall \tau\not\in I\]
says that $\tau\notin I$ implies $\tau\not\in J$, i.e., $J\subseteq I$.
The first condition of $\cU_I(\epsilon)$
\[\sum_{\tau\in I} \tilde \deg_\tau y \geq \sum_{\tau\in I} \tilde x_{I,\tau}+\epsilon\]
then says that $J\neq I$.

For each $J\subsetneq I$, let $\cV_J$ be a strict neighborhood of $\underline{\deg}^{-1} x_J$ on which $f$ is defined.
Moreover we can choose $\cV_J$ in the form 
$$\cV_J = \{y\in \cY\colon \deg_\tau y \leq \epsilon_\tau \text{ if } \tau \in J, \deg_\tau y \geq 1-\epsilon_\tau, \text{ if } \tau \not\in J\},$$
for some rational $\epsilon_\tau >0$.
On the other hand, let $\epsilon'_\tau<\epsilon_\tau$ be a rational number, and define
\[
\cV = \left\{y\in \cY\colon 
\begin{array}{l}
\deg_\tau y \geq \epsilon'_\tau \text{ if } \tau \in I, \deg_\tau y \leq 1-\epsilon'_\tau, \text{ if } \tau \not\in I;\\
\sum_{\tau\in I} \tilde \deg_\tau y \geq \sum_{\tau\in I} \tilde x_{I,\tau}+\epsilon, \tilde\deg_\tau y \geq p^{g-2}+\cdots+1+\epsilon, \forall \tau\not\in I
\end{array}
\right\}.\]
Because $\cV_J$'s and $\cV$ are defined by a finite number of affine non-strict inequalities with rational coefficients, they are quasi-compact admissible opens of $\cU_I(\epsilon)$.
We hence have an admissible cover $\cU_I(\epsilon)=\bigcup_{J\subsetneq I}\cV_J \cup \cV$.

Since $\cV$ is disjoint from $\underline{\deg}^{-1} x_J$ for any $J\subsetneq I$ from its definition,  $U_p$ strictly increases $\sum_{\tau\in\Sigma} \tilde\deg_\tau$ on $\cV$.
Using the Maximum Modulus Principle, the quasi-compactness of $\cV$ implies that there is a positive lower bound for the increase of $\sum_{\tau\in \Sigma}\tilde\deg_\tau$ under $U_p$ on $\cV$.
Because $\cU_I(\epsilon)$ is $U_p$-stable, there exists $M>0$ such that $U_p^M \cV\subseteq \bigcup_{J\subsetneq I}\cV_J$.
Since $f$ is defined on $\bigcup_{J\subsetneq I}\cV_J$, we may define $f$ on $\cV$ by $(\frac{U_p}{a_p})^Mf$.
On the intersection $(\bigcup_{J\subsetneq I}\cV_J)\cap \cV$, the definitions of $f$ coincide since $a_p$ is the $U_p$-eigenvalue of $f$.
We can then define $f$ on the whole $\cU_I(\epsilon)$ through the admissible cover $\cU_I(\epsilon)=\bigcup_{J\subsetneq I}\cV_J \cup \cV$.
\end{proof}

\subsection{Analytic continuation near vertices} \label{subsec:vertex}

In this subsection, we will make use of the small slope assumption \eqref{small_slope} to extend $f$ to a strict neighborhood of $\deg^{-1}x_I$.

Let's first give an outline of the strategy.
By \eqref{small_slope}, for any small enough  $\epsilon>0$ we have
\begin{align} \label{small_slope'}
\val_p (a_p) \leq \inf_{\tau\in I} k_\tau-g-\epsilon\sum_{\tau\in I}k_\tau.
\end{align}
Possibly making it smaller, we will first fix such a rational number $\epsilon$.
Then we will choose a rational number $\delta>0$ based on $\epsilon$, and define a sequence of strict neighborhoods 
\[
S_{I,0}(\delta)\supseteq S_{I,1}(\delta) \supseteq \cdots
\]
of $\deg^{-1}x_I$.
When $\delta'<\delta$ we will show that $S_{I,m}(\delta') \subsetneq S_{I,m}(\delta)$.
We have extended $f$ to $\cU_I(\delta)$ by Lemma~\ref{lem:loc_u}.
Further applying some power of $\frac{U_p}{a_p}$, we will be able to extend $f$ to $S_{I,0}(\delta) \setminus S_{I,m}(\delta')$, named $f_m$.
We will also define $F_m$ on $S_{I,m}(\delta)$.
With the help of the estimates in Section \ref{subsec:norm}, we will show that when $m\to\infty$, $f_m$ and $F_m$ glue to define an extension of $f$ on $S_{I,0}(\delta)$.

To begin, we prove the following lemma regarding the twisted directional degrees of points in the set $U_p(y)$, when $y\in \cY$ satisfies $\deg y = x_I$. The lemma will be used to decompose the $U_p$-correspondence $\cY(p)$ over $S_{I,1}(\delta)$ into the special part $\cY(p)^{sp}$ and the non-special part $\cY(p)^{nsp}$, and so the $U_p$-operator becomes $U_p^{sp}+U_p^{nsp}$.
\begin{lemma}
\label{lem:deg_nsp}
Let $y=(\underline A,H)\in \cY$.
Let $y_1=(\underline{A/H_1},\bar{H} = A[p]/H_1)$ and $y_2=(\underline{A/H_2},\bar{H}=A[p]/H_2)$ be in $U_p(y)$ and $y_1\neq y_2$.
\begin{itemize}
    \item[i.] If $y,y_1\in \underline{\tilde\deg}^{-1} \tilde x_I$ for some $I\subseteq\Sigma$,
    then
    \[  
    \tilde\deg_\tau H_2 = \inf (\tilde\deg_\tau H, \tilde\deg_\tau H_1), \text{ for all }\tau\in\Sigma.
    \]
    \item[ii.] There exists arbitrarily small positive rational number $\epsilon$ so that
    if $|\tilde\deg_\tau(y)-\tilde x_{I,\tau}|\leq\epsilon$ and $|\tilde\deg_\tau(y_1)-\tilde x_{I,\tau}|\leq\epsilon$ for some $I\subseteq \Sigma$, then
    \[  
    \tilde\deg_\tau H_2 = \inf (\tilde\deg_\tau H, \tilde\deg_\tau H_1), \text{ for all }\tau\in\Sigma.
    \]
    In particular,  $y_2\in \cU_\varnothing(\epsilon)$.

\end{itemize}
\end{lemma}
\begin{proof}
For the proof of \textit{i.}, see {\cite[Lemma 5.1.5 1.]{Kas}}
The first statement of \textit{ii.}\ follows from {\cite[Lemma 5.1.5 2(a)]{Kas}}.

The only statement remained to be proved is the one after ``In particular''.
By assumption, 
\begin{align*}
\tilde \deg_\tau H_2 
= \inf (\tilde \deg_\tau H, \tilde \deg_\tau H_1)
&= \begin{cases} \tilde\deg_\tau H & \text{ if }\tau \in I\\
\tilde\deg_\tau H_1 & \text{ if }\tau\not\in I
\end{cases}
\end{align*}
and 
\begin{align*}
\tilde \deg_\tau y_2 
= (p^{g-1}+\cdots+1) -\tilde \deg_\tau H_2
&= \begin{cases} (p^{g-1}+\cdots+1)-\tilde\deg_\tau H & \tau \in I\\
(p^{g-1}+\cdots+1)-\tilde\deg_\tau H_1 & \tau\not\in I.
\end{cases}\\
&\geq \begin{cases} (p^{g-1}+\cdots+1)-\tilde x_{I,\tau}-\epsilon & \tau \in I\\
\tilde x_{I,\tau}-\epsilon & \tau\not\in I
\end{cases}\\
&\geq p^{g-1}-\epsilon
\end{align*}
If we further require that $\epsilon<\frac12(p^{g-1}-p^{g-2}-\cdots-1)$, then $\tilde\deg_\tau y_2 \geq p^{g-2}+\cdots+1+\epsilon$, i.e., $y_2\in \cU_\varnothing(\epsilon)$.

\end{proof}
\begin{corollary} \label{cor}
Let $I\subseteq \Sigma$ and $I\neq \varnothing$.
Let $\epsilon$ be a rational number as in Lemma \ref{lem:deg_nsp} \emph{ii.} such that $\epsilon<\frac12(p^{g-1}-p^{g-2}-\cdots-1)$.
Let $y\in \cY$ be such that $|\tilde\deg_\tau(y)-\tilde x_{I,\tau}|\leq\epsilon$ for all $\tau\in\Sigma$.
Then there exists at most one point $y_1\in U_p(y)$ such that $|\tilde\deg_\tau(y_1)-\tilde x_{I,\tau}|\leq\epsilon$ for all $\tau\in\Sigma$.
\end{corollary}
\begin{proof}
By the proof of Lemma \ref{lem:deg_nsp} \emph{ii.}, if $y_2\in U_p(y)$ and $y_2\neq y_1$, then $\tilde\deg_\tau(y_2) \geq p^{g-1}-\epsilon$ for all $\tau \in \Sigma$.
Since $I\neq \varnothing$, we pick an arbitrary $\tau_0\in I$.
Then $$\tilde\deg_{\tau_0}(y_2)-x_{I,\tau_0}\geq  (p^{g-1}-\epsilon)-(p^{g-2}+\cdots+1) > \epsilon.$$
\end{proof}



For any rational number $\delta>0$, consider the strict neighborhood of $\underline\deg^{-1} x_I$: 
\[
S_{I,0}(\delta) 
\defeq \left\{ y\in \cY\colon \sum_{\tau\in I}\tilde\deg_\tau y\leq \sum_{\tau\in I} \tilde x_{I,\tau}+\delta, \tilde\deg_\tau y \geq \tilde x_{I,\tau}-\delta, \forall \tau\not\in I \right\},
\]
which is a quasi-compact admissible open.
Recall from Section \ref{sec:Up} that the $U_p$-correspondence is given by $p_1\colon \cY(p)\rightarrow \cY, (\underline{A},H,H_1)\mapsto (\underline{A},H)$ and $p_2\colon \cY(p)\rightarrow \cY, (\underline{A},H,H_1)\mapsto (\underline{A}/H_1,\bar{H})$.
Define 
\[
S_{I,1}(\delta) \defeq  p_1(p_1^{-1}S_{I,0}(\delta)\cap p_2^{-1}S_{I,0}(\delta)),
\]
which is a quasi-compact admissible open of $\cY$ because it is the pushforward of a quasi-compact admissible open by the finite \'etale morphism $p_1$.
Note that 
\[
S_{I,1}(\delta)=\{y\in S_{I,0}(\delta): \exists y_1\in U_p(y) \text{ also in $S_{I,0}(\delta)$}\},
\]
so $S_{I,1}(\delta)$ is called the \emph{special locus of order $1$} in $S_{I,0}(\delta)$.

Let $\epsilon$ be a small enough rational number as in Lemma \ref{lem:deg_nsp} \emph{ii.} such that $\epsilon<\frac12(p^{g-1}-p^{g-2}-\cdots-1)$, and that the small slope condition \eqref{small_slope'} is satisfied.
Note that the $S_{I,0}(\delta)$'s contain a fundamental system of strict neighborhoods of $\underline{\deg}^{-1}x_I$.
Hence we choose a rational number $\delta>0$ so that $S_{I,0}(\delta)\subseteq \{y\in \cY\colon |\tilde\deg_\tau y-\tilde x_{I,\tau}|<\epsilon\}$ and $S_{I,0}(\delta)\subseteq \{y\in \cY\colon |\deg_\tau y- x_{I,\tau}|<\epsilon\}$.
With this choice of $\delta$, we see by Corollary \ref{cor} that the $y_1$ in the definition of $S_{I,1}(\delta)$ is unique.

Hence we have a correspondence $\cY(p)^{sp}\defeq p_1^{-1}S_{I,0}(\delta)\cap p_2^{-1}S_{I,0}(\delta)\subseteq \cY(p)$
\begin{center}
\begin{tikzcd}
 & \cY(p)^{sp} \ar[ld, "p_1^{sp}"'] \ar[rd, "p_2^{sp}"] & \\
S_{I,1}(\delta)&& S_{I,0}(\delta)   
\end{tikzcd}
\end{center}
where $p_i^{sp}$ is the restriction of $p_i$ to $\cY(p)^{sp}$, and $p_1^{sp}$ is an isomorphism.
Then as before in Section \ref{sec:Up}, for any subset $\cU\subseteq S_{I,1}(\delta)$, let $U_p^{sp}(\cU)\defeq p_2^{sp}(p_1^{sp})^{-1}(\cU)$.
If $\cU\subseteq S_{I,0}(\delta)$ is an admissible open, then $(U_p^{sp})^{-1}\cU = p_1^{sp}(p_2^{sp})^{-1}\cU$ is also an admissible open because $p_1^{sp}$ is finite \'etale (indeed an isomorphism).
For $f\in \omega^{\underline{k}}(\cU)$, let $U_p^{sp}f\in \omega^{\underline{k}}((U_p^{sp})^{-1}\cU)$ be $(U_p^{sp}f)(\underline{A},H) \defeq \frac1{p^g} \pr^\ast f(\underline{A}/H_1,\bar{H})$, where $H_1$ is such that $p_1^{sp}(\underline{A},H,H_1) = (\underline{A},H)$.

We also define $\cY(p)^{nsp}\defeq \left(\cY(p)\times_{\cY,p_1} S_{I,1}(\delta)\right)\setminus \cY(p)^{sp}$.
By Lemma \ref{lem:deg_nsp} \emph{ii.}, we have
$p_2(\cY(p)^{nsp})\subseteq  \cU_\varnothing(\epsilon)$.
Hence 
\begin{center}
\begin{tikzcd}
 & \cY(p)^{nsp} \ar[ld, "p_1^{nsp}"'] \ar[rd, "p_2^{nsp}"] & \\
S_{I,1}(\delta)&& \cU_\varnothing(\epsilon)   
\end{tikzcd}
\end{center}
where $p_i^{nsp}$ is again the restriction of $p_i$.
We similarly define $U_p^{nsp}$ on subsets $\cU\subseteq S_{I,1}(\delta)$ and on $f\in \omega^{\underline{k}}(\cU)$ when $\cU\subseteq S_{I,0}(\delta)$ is an admissible open.

Define the quasi-compact admissible open
\[\cV_{I}(\delta) = \{y\in\cY\colon \sum_{\tau\in I} \tilde \deg_\tau y \geq \sum_{\tau\in I} \tilde x_{I,\tau}+\delta, \tilde\deg_\tau y \geq \tilde x_{I,\tau} -\delta, \forall \tau\not\in I\}.\]
Then $S_{I,0}(\delta)\cup \cV_{I}(\delta) $ is $U_p$-stable because $U_p$ increases twisted directional degrees (Proposition~\ref{prop:twisted_deg_increase}(1)).
Hence we have
\[
U_p(S_{I,0}(\delta) \setminus S_{I,1}(\delta))\subseteq \cV_I(\delta).
\]
Note that $\cV_I(\delta)\subseteq \cU_I(\delta)$, and the latter was defined in Section \ref{subsec:auto}.

\begin{lemma} \label{lem:strict_nbd}
Let $\delta'<\delta$ be two positive rational numbers.
Then $S_{I,1}(\delta)$ is a strict neighborhood of  $S_{I,1}(\delta')$. 
\end{lemma}
\begin{proof}

Because $S_{I,0}(\delta)$ is defined by inequalities of twisted directional degrees, when $\delta'<\delta$ are two positive rational numbers, then $S_{I,0}(\delta)$ is a strict neighborhood of  $S_{I,0}(\delta')$.
By definition, 
$S_{I,1}(\delta) =  p_1(p_1^{-1}S_{I,0}(\delta)\cap p_2^{-1}S_{I,0}(\delta))$.
Since $p_1$ is finite \'etale, pushforward by $p_1$ preserves quasi-compact admissible opens, and hence $S_{I,1}(\delta)$ is a strict neighborhood of  $S_{I,1}(\delta')$.

\end{proof}


As explained above, for any admissible open $\cU\subseteq S_{I,0}(\delta)$, $(U_p^{sp})^{-1}\cU$ is also an admissible open.
Define the admissible open
\[
S_{I,m}(\delta) = (U_p^{sp})^{-m}S_{I,0}(\delta),
\]
which is quasi-compact because $S_{I,0}(\delta)$ is.
Lemma~\ref{lem:strict_nbd} says that if $\delta'<\delta$ are two positive rational numbers, then $S_{I,1}(\delta)$ and $S_{I,0}(\delta)\setminus S_{I,1}(\delta')$ form an admissible covering of $S_{I,0}(\delta)$.
Then $S_{I,m}(\delta)$ and $S_{I,0}(\delta)\setminus S_{I,m}(\delta')$ also form an admissible covering of $S_{I,0}(\delta)$.

Now we are ready to prove analytic continuation near vertices.
\begin{proposition} \label{prop:an_vertices}
Let $f$ be an overconvergent Hilbert modular form of weight $\underline{k}$.
Let $I\subseteq \Sigma$.
Suppose that $f$ is defined on a strict neighborhood of $\underline{\deg}^{-1} x_J$ for all $J\subsetneq I$.
Let $\epsilon$ be a small enough rational number as in Lemma \ref{lem:deg_nsp} \emph{ii.} such that $\epsilon<\frac12(p^{g-1}-p^{g-2}-\cdots-1)$, and that 
\begin{align*} 
\val_p (a_p) \leq \inf_{\tau\in I} k_\tau-g-\epsilon\sum_{\tau\in I}k_\tau.
\end{align*}
Let $\delta>0$ be a rational number so that $S_{I,0}(\delta)\subseteq \{y\in \cY\colon |\tilde\deg_\tau y-\tilde x_{I,\tau}|<\epsilon\}$ and $S_{I,0}(\delta)\subseteq \{y\in \cY\colon |\deg_\tau y- x_{I,\tau}|<\epsilon\}$.
Then $f$ can be extended to $S_{I,0}(\delta)$, which is a strict neighborhood of $\underline{\deg}^{-1}x_I$.
\end{proposition}
\begin{proof}
By definition, $S_{I,m-1}(\delta)\supseteq S_{I,m}(\delta)$.
In addition, $U_p^m(S_{I,0}(\delta) \setminus S_{I,m}(\delta)) \subseteq \cV_I(\delta)$.
By Lemma~\ref{lem:loc_u}, we can extend $f$ to $\cV_I(\delta)\subseteq \cU_I(\delta)$.
Then we can further extend $f$ by $(\frac{U_p}{a_p})^mf$ to $(U_p)^{-m}\cV_I(\delta) \supseteq S_{I,0}(\delta) \setminus S_{I,m}(\delta)$.
Similarly, for any other rational number $\delta'<\delta$, we can extend $f$ by $(\frac{U_p}{a_p})^mf$ to $(U_p)^{-m}\cV_I(\delta') \supseteq S_{I,0}(\delta') \setminus S_{I,m}(\delta')$.
Because $S_{I,0}(\delta) \setminus S_{I,m}(\delta)$ and $S_{I,0}(\delta') \setminus S_{I,m}(\delta')$ form an admissible covering of $S_{I,0}(\delta) \setminus S_{I,m}(\delta')$, we can actually extend $f$ to $S_{I,0}(\delta) \setminus S_{I,m}(\delta')$.

We denote by $f_m$ the extension of $f$ to $S_{I,0}(\delta) \setminus S_{I,m}(\delta')$.

On the other hand, by Lemma~\ref{lem:loc_u}, we can extend $f$ to $\cU_\varnothing(\epsilon)$. 
Then 
\[
F_m \defeq \sum_{j=0}^{m-1} (\frac1{a_p})^{j+1}U_p^{nsp}(U_p^{sp})^jf
\]
can be defined on $(U_p^{sp})^{-(m-1)}(U_p^{nsp})^{-1}(\cU_\varnothing(\epsilon)) \supseteq S_{I,m}(\delta)$. 

Assume the norm estimates in Proposition~\ref{prop:norm} in the next subsection.
By (2), we can choose a subsequence so that $F_m \bmod {p^m}$ and $f_m \bmod {p^m}$ glue as $h_m$ (only defined modulo $p^m$) under the admissible covering $S_{I,0}(\delta)\setminus S_{I,m}(\delta')$ and $S_{I,m}(\delta)$ of $S_{I,0}(\delta)$.
We have $h_m \equiv f \pmod {p^m}$ on $S_{I,0}(\delta)\setminus S_{I,m}(\delta')$.
By (3), we can further choose a subsequence so that $h_{m+1} \bmod {p^m}$ agrees with $h_m \bmod{p^m}$ on $S_{I,m+1}(\delta)$.
Hence $h=\lim_{m\to\infty}h_m$ is defined on $S_{I,0}(\delta)$, and $h=f$ on $S_{I,0}(\delta) \setminus \bigcap_m S_{I,m}(\delta')$.
Hence $h$ is the desired extension of $f$ to $S_{I,0}(\delta)$.
\end{proof}

\subsection{Norm estimates} \label{subsec:norm}
Assume that $\val_p (a_p) \leq \inf_{\tau\in I} k_\tau-g-\epsilon\sum_{\tau\in I}k_\tau$.
Choose a rational number $\delta>0$ so that $S_{I,0}(\delta)\subseteq \{y\in \cY\colon |\tilde\deg_\tau y-\tilde x_{I,\tau}|<\epsilon\}$ and $S_{I,0}(\delta)\subseteq \{y\in \cY\colon |\deg_\tau y- x_{I,\tau}|<\epsilon\}$.
Also let $\delta'<\delta$ be another positive rational number.

Let $f_m$ defined on $S_{I,0}(\delta)$ and $F_m$ defined on $S_{I,0}(\delta)\setminus S_{I,m}(\delta')$ as in the previous section.
The following proposition records the norm estimates used to glue $f_m$ and $F_m$ in the previous section.
\begin{proposition} \label{prop:norm} \hfill
\begin{enumerate} 
    \item $|F_m|_{S_{I,m}(\delta)}$ and $|f_m|_{S_{I,0}(\delta)\setminus S_{I,m}(\delta')}$ are bounded.
    \item $|F_m-f_m|_{S_{I,m}(\delta)\setminus S_{I,m}(\delta')} \to 0$. 
    \item $|F_{m+1} - F_m|_{S_{I,m+1}(\delta)} \to 0$.
\end{enumerate}
\end{proposition}

We need the following two lemmas to prove Proposition~\ref{prop:norm}.
\begin{lemma} \label{lem:keynormest}
Let $\cV\subseteq S_{I,1}(\delta)$ and $h\in \omega^{\underline{k}}(U_p^{sp}(\cV))$.
Then
\[
|U_p^{sp}(h)|_{\cV} \leq p^{g-\sum_{\tau\in I}k_\tau(1-\epsilon)} |h|_{U_p^{sp}(\cV)}.
\]
In particular, if $\val_p (a_p) < \inf_{\tau\in I}k_\tau -g -\epsilon\sum_{\tau\in I} k_\tau$, then
\[|\frac{U_p^{sp}}{a_p}h|_{\cV}\leq p^{-\mu} |h|_{U_p^{sp}(\cV)}\]
for some small enough $\mu>0$.
\end{lemma}
\begin{proof}
Recall that $U_p$ is defined by
\[
H^0(U_p(\cV),\omega^{\underline{k}}) \rightarrow H^0(p_1^{-1}(\cV),p_2^\ast\omega^{\underline{k}}) \xrightarrow{\pi^\ast} H^0((p_1^{-1}(\cV)),p_1^\ast\omega^{\underline{k}}) \xrightarrow{\frac1{p^g}\Tr_{p_1}} H^0(\cV,\omega^{\underline{k}}).
\]
Let $y\in \cV$ and $y_1\in U_p^{sp}(y)$.
Then 
\begin{align*}
|(U_p^{sp}h)(y)| = |\frac1{p^g} (\pi^\ast h)(y_1)| = p^{g -\sum_{\tau\in\Sigma}k_\tau\deg_\tau H_1} |h(y_1)|.
\end{align*}
By assumption, $y_1\in \cV\subseteq S_{I,0}(\delta)$, i.e., $\sum_{\tau\in I}\tilde \deg_\tau y_1 \leq \sum_{\tau \in I} \tilde x_I +\delta, \tilde\deg_\tau y_1 \geq \tilde x_I -\delta$.
Hence by our choice of $\delta$, 
we have 
\[
|\deg_\tau y_1 - x_{I,\tau}| < \epsilon,
\]
namely
\[
|\deg_\tau H_1 - x_{I^c,\tau}| < \epsilon.
\]
Then
\[
g -\sum_{\tau\in\Sigma}k_\tau\deg_\tau H_1 
\leq g - \sum_{\tau\in I} k_\tau (1-\epsilon).
\]

\end{proof}

\begin{lemma} \label{lem:F+f}
For $1\leq j \leq m$, 
$f_m - (\frac{U_p^{sp}}{a_p})^j f_m = F_j$ on $S_{I,j}(\delta)\setminus S_{I,m}(\delta')$.
\end{lemma}
\begin{proof}
Recall that we have fixed $\delta'<\delta$, and $f_m$ is defined on $S_{I,0}(\delta)\setminus S_{I,m}(\delta')$. 
In particular, $(\frac{U_p^{sp}}{a_p})^j f_m$ is defined on $(U_p^{sp})^{-j}S_{I,0}(\delta)\setminus S_{I,m}(\delta') = S_{I,j}(\delta)\setminus S_{I,j+m}(\delta')$.

By definition, $F_j = \sum_{\ell=0}^{j-1} (\frac1{a_p})^{\ell+1} U_p^{nsp} (U_p^{sp})^\ell f$ on $S_{I,j}(\delta)$. 
Hence $F_j + (\frac{U_p^{sp}}{a_p})^{j}f_m $ is defined on $S_{I,j}(\delta)\setminus S_{I,m}(\delta')$.
A simple calculation using the fact that $U_p = U_p^{sp}+ U_p^{nsp}$ yields the claimed equality $F_j + (\frac{U_p^{sp}}{a_p})^{j}f_m = f_m$.
\end{proof} 

\begin{proof}[Proof of Proposition~\ref{prop:norm}]\hfill
\begin{enumerate}
    \item
Because $f$ is defined on the quasi-compact open $\cV_I(\delta)$, $|f|_{\cV_I(\delta)}$ is bounded.
Since $U_p$ is a compact operator, $|f_1|_{S_{I,0}(\delta)\setminus S_{I,1}(\delta)} \leq |\frac{U_p}{a_p} f|_{\cV_I(\delta)}$ is also bounded.
Similarly, $|f_1|_{S_{I,0}(\delta')\setminus S_{I,1}(\delta')}$ is bounded, and hence $|f_1|_{S_{I,0}(\delta)\setminus S_{I,1}(\delta')}$ is bounded.

We will show that $|f_m|_{S_{I,0}(\delta)\setminus S_{I,m}(\delta')}\leq \sup(|f_1|_{S_{I,0}(\delta)\setminus S_{I,1}(\delta')}, |F_1|_{S_{I,m}(\delta)})$ for all $m\geq1$.
Because $f_m$'s are compatible, it suffices to show that 
\[
|f_m|_{S_{I,m}(\delta)\setminus S_{I,m}(\delta')}\leq \sup(|f_1|_{S_{I,0}(\delta)\setminus S_{I,1}(\delta')}, |F_1|_{S_{I,m}(\delta)})
\] 
for all $m\geq1$.
We do this by induction on $m$.
By Lemma~\ref{lem:F+f}, $f_m - \frac{U_p^{sp}}{a_p} f_m = F_1$ on $S_{I,1}(\delta)\setminus S_{I,m}(\delta')$.
Then it suffices to show that 
\[
|\frac{U_p^{sp}}{a_p} f_m |_{S_{I,m}(\delta)\setminus S_{I,{m+1}}(\delta')} \leq \sup(|f_1|_{S_{I,0}(\delta)\setminus S_{I,1}(\delta')}, |F_1|_{S_{I,m}(\delta)}).
\]
By Lemma~\ref{lem:keynormest}, 
\begin{align*}
|\frac{U_p^{sp}}{a_p} f_m |_{S_{I,m}(\delta)\setminus S_{I,{m+1}}(\delta')} 
&\leq |f_{m}|_{S_{I,m-1}(\delta)\setminus 
S_{I,{m}}(\delta')}\\
&= |f_{m-1}|_{S_{I,m-1}(\delta)\setminus S_{I,{m}}(\delta')}
\end{align*}
Hence 
\[|\frac{U_p^{sp}}{a_p} f_m |_{S_{I,m}(\delta)\setminus S_{I,{m+1}}(\delta')} \leq \sup(|f_1|_{S_{I,0}(\delta)\setminus S_{I,1}(\delta')}, |F_1|_{S_{I,m}(\delta)})\]
by induction hypothesis.

As for $|F_m|_{S_{I,m}(\delta)}$, by Lemma~\ref{lem:keynormest},
\begin{align*}
|F_m|_{S_{I,m}(\delta)}
&\leq \sup_{0\leq j\leq m-1} |(\frac1{a_p})^{j+1} U_p^{nsp} (U_p^{sp})^jf|_{S_{I,m}(\delta)} \\
&= \sup_{0\leq j\leq m-1} |(\frac{U_p^{sp}}{a_p})^j F_1|_{S_{I,m}(\delta)}\\
&\leq \sup_{0\leq j\leq m-1} |F_1|_{S_{I,m-j}(\delta)} \\
&= |F_1|_{S_{I,1}(\delta)}.
\end{align*}

\item
By Lemma~\ref{lem:F+f} and Lemma~\ref{lem:keynormest},
\begin{align*}
|F_m-f_m|_{S_{I,m}(\delta)\setminus S_{I,m}(\delta')} 
&=|(\frac{U_p^{sp}}{a_p})^m f_m|_{S_{I,m}(\delta)\setminus S_{I,m}(\delta')} \\
&\leq p^{-m \mu}|f_m|_{S_{I,0}(\delta)\setminus S_{I,0}(\delta')} \\
&=p^{-m \mu}|f_0|_{S_{I,0}(\delta)\setminus S_{I,0}(\delta')} \\
&\to 0 \text{ as $m\to \infty$}.
\end{align*}

\item
By Lemma~\ref{lem:keynormest},
\begin{align*}
|F_{m+1}-F_m|_{S_{I,m+1}(\delta)} 
&=|(\frac1{a_p})^{m+1} U_p^{nsp}(U_p^{sp})^m f|_{S_{I,m+1}(\delta)}  \\
&=|(\frac{U_p^{sp}}{a_p})^m F_1|_{S_{I,m+1}(\delta)} \\ 
&\leq p^{-m \mu}|F_1|_{S_{I,1}(\delta)}  \\
&\rightarrow 0 \text{ as $m\to \infty$}.
\end{align*}
\end{enumerate}
\end{proof}

\subsection{Finishing the proof of Theorem~\ref{thm:partial_classicality}}
Following the paragraph just before Section~\ref{subsec:auto},  we assume that the overconvergent form $f$ is defined on a strict neighborhood of $\underline{\deg}^{-1}x_J$ for all $J\subsetneq I$.
We also assume that $f$ satisfies the small slope condition \eqref{small_slope}. 
Let $\epsilon$ be a small enough rational number as in Lemma \ref{lem:deg_nsp} \emph{ii.} such that $\epsilon<\frac12(p^{g-1}-p^{g-2}-\cdots-1)$, and that 
\begin{align*} 
\val_p (a_p) \leq \inf_{\tau\in I} k_\tau-g-\epsilon\sum_{\tau\in I}k_\tau.
\end{align*}
By Proposition~\ref{prop:an_vertices} we can extend $f$ to a strict neighborhood $S_{I,0}(\delta)$ of  $\underline{\deg}^{-1}x_I$ for any small enough rational number $\delta>0$.

Note that the vertices in $\cF_I$ are exactly the $x_J$'s with $J\subseteq I$, so we have extended $f$ to a strict neighborhood of the inverse image of $\underline{\deg}$ of all the vertices of $\cF_I$.
We will show that $f$ can be extended to a strict neighborhood $\cU$ of $\cY\cF_I$, again using the argument in Lemma~\ref{lem:loc_u} that $U_p$ strictly increases the sum of twisted directional degrees when the $\underline{\deg}$ is not one of the vertices of $[0,1]^g$.

Define a quasi-compact admissible open $$\cU = \{y\in \cY \colon \tilde\deg_\tau y \geq p^{g-2}+\cdots+1+\epsilon, \forall \tau\not\in I\}.$$
Recall that $\cY\cF_I = \{y\in \cY\colon \deg_\tau y = 1, \forall \tau \not\in I\}$. If $y\in \cY\cF_I$, then for $\tau\not\in I$, 
$$\tilde\deg_\tau y \geq p^{g-1} > p^{g-2}+\cdots +1+\epsilon.$$
Hence $\cU$ is a strict neighborhood of $\cY\cF_I$.
We have shown in the proof of Lemma~\ref{lem:loc_u} that the condition of $\cU$ implies that if $y\in \cU$ is such that $\underline{\deg}(y) = x_J$ for some $J\subseteq \Sigma$, then $J\subseteq I$.
Moreover, $\cU$ is $U_p$-stable because $U_p$ increases twisted directional degrees (Proposition~\ref{prop:twisted_deg_increase}(1)).

For each $J\subseteq I$, let $\cV_J$ be a strict neighborhood of $\underline{\deg}^{-1} x_J$ on which $f$ is defined, and we explicitly choose $\cV_J$ in the form 
$$\cV_J = \{y\in \cY\colon \deg_\tau y \leq \epsilon_\tau \text{ if } \tau \in J, \deg_\tau y \geq 1-\epsilon_\tau, \text{ if } \tau \not\in J\},$$
for some rational $\epsilon_\tau >0$.
Let $\epsilon'_\tau<\epsilon_\tau$ be a rational number, and define the quasi-compact admissible open
\[
\cV = \left\{y\in \cY\colon 
\begin{array}{l}
\deg_\tau y \geq \epsilon'_\tau \text{ if } \tau \in I, \deg_\tau y \leq 1-\epsilon'_\tau, \text{ if } \tau \not\in I;\\
 \tilde\deg_\tau y \geq p^{g-2}+\cdots+1+\epsilon, \forall \tau\not\in I
\end{array}
\right\}.\]
We have an admissible cover $\cU=\bigcup_{J\subseteq I}\cV_J \cup \cV$.

Since $\cV$ is disjoint from $\underline{\deg}^{-1} x_J$ for any $J\subseteq I$ from its definition,  $U_p$ strictly increases $\sum_{\tau\in\Sigma} \tilde\deg_\tau$ on $\cV$ by Proposition~\ref{prop:twisted_deg_increase}(2).
Using the Maximum Modulus Principle, the quasi-compactness of $\cV$ implies that there is a positive lower bound for the increase of $\sum_{\tau\in \Sigma}\tilde\deg_\tau$ under $U_p$ on $\cV$.
Because $\cU$ is $U_p$-stable, there exists $M>0$ such that $U_p^M \cV\subseteq \bigcup_{J\subseteq I}\cV_J$.
Since $f$ is defined on $\bigcup_{J\subseteq I}\cV_J$, we may define $f$ on $\cV$ by $(\frac{U_p}{a_p})^Mf$.
On the intersection $(\bigcup_{J\subseteq I}\cV_J)\cap \cV$, the definitions of $f$ coincide since $a_p$ is the $U_p$-eigenvalue of $f$.
We can then define $f$ on the whole $\cU$ through the admissible cover $\cU=\bigcup_{J\subseteq I}\cV_J \cup \cV$.


\bibliographystyle{amsalpha}
\bibliography{bibliography}

\end{document}